\documentclass[11pt]{article}
\usepackage[a4paper,margin=2.6cm]{geometry}
\usepackage{amsmath,amssymb,amsthm}
\usepackage{booktabs}

\newtheorem{theorem}{Theorem}
\newtheorem{lemma}{Lemma}
\newtheorem{proposition}{Proposition}
\newtheorem{corollary}{Corollary}
\newtheorem{definition}{Definition}
\theoremstyle{remark}
\newtheorem{remark}{Remark}

\newcommand{\LB}{\mathrm{LB}}
\newcommand{\bg}{\mathrm{bg}}
\newcommand{\KS}{\mathrm{KS}}
\newcommand{\refi}{\mathrm{ref}}
\newcommand{\vput}{\mathrm{put}}
\newcommand{\diam}{\operatorname{diam}}

\title{A certified refinement and asymptotic analysis\\
of the Kuznetsov--Sahinidis diameter bound\\
for Lennard--Jones clusters}
\author{Guillaume Lecomte\thanks{Corresponding author. All results were independently regenerated and checked by the author.}}
\date{July 9, 2026}

\begin{document}
\maketitle

\begin{abstract}
Kuznetsov and Sahinidis (\emph{J.\ Glob.\ Optim.}, 2025) prove distance bounds that confine optimal Lennard--Jones clusters and shrink the search region of deterministic solvers; their diameter bound charges each unit-width layer the loosest internal energy $-\binom n2$. We replace this by a certified estimate built from their own subset inequality and the proven minima $V^*_5$, $V^*_6$, and minimize the resulting layer bound over population profiles. Centred decreasing profiles supply candidate minima; an arrangement-free relaxation, whose only classical ingredient is the rearrangement inequality for sequences, closes every certificate over all profiles; and all comparisons are re-verified in directed-rounding arithmetic. For $5\le N\le 200$ this certifies a one-layer improvement of the published bound at 92 sizes. The improvement resolves no open global-optimization case: it is a rigorous tightening of a published a priori bound, with a precise account of the mechanism. A direct downstream test on the only tractable sizes ($N\le 6$) finds the diameter box is not the binding resource for a deterministic solver there, and no solver runs at the sizes the refinement affects ($N\ge 38$); we therefore present the result as a theoretical note. We also derive the asymptotic form of the bound, $\rho_\KS = N-\Theta(\sqrt N)$, resolving a point left open by Kuznetsov and Sahinidis; the gain of the refinement itself grows like $\Theta(\sqrt N)$ layers, with an exact asymptotic constant.
\end{abstract}

\section{Introduction}\label{sec:intro}

The Lennard--Jones (LJ) cluster problem asks for the arrangement of $N$ identical atoms in $\mathbb R^3$ that minimizes the pairwise potential energy
\begin{equation}\label{eq:pot}
V_{\mathrm{LJ}}(r) = r^{-12} - 2\,r^{-6}, \qquad V_{\mathrm{LJ}}(1) = -1 = \min_r V_{\mathrm{LJ}},
\end{equation}
summed over all pairs, distances being measured in units of the equilibrium pair distance. Finding the global minimum is a classical benchmark in molecular conformation and a hard nonconvex global-optimization problem. While thousands of putative global minima are tabulated from heuristic search, deterministic global optimization, which returns a numerical certificate of optimality, has succeeded only for the very smallest clusters.

A key ingredient of any deterministic approach is an a priori bound on the diameter of an optimal cluster: it restricts the spatial box in which atoms may lie and thus the branch-and-bound search space. Kuznetsov and Sahinidis \cite{ks2025} give such a bound through an elegant layer argument. The present note refines it. We keep the KS layer construction, change the energy charged to a single layer, and use the conservative inter-layer separation $d-1$ when assigning the pairwise floor (Section~\ref{sec:ks}). Charging each layer a certified subset-based estimate, rather than the trivial all-pairs value, yields a strictly tighter diameter bound for 92 sizes in the range $15\le N\le 200$. The per-layer refinement uses only the KS subset inequality, the certified values $V^*_5$ \cite{vanaret2014} and $V^*_6$ \cite{ks2025}, and explicit upper-bound configurations through $V_\vput(N)$. We state the main limitation at the outset: this note certifies no new global minimum for any $N$; it tightens an a priori geometric bound.

The technical content is concentrated in two elementary reductions. Replacing the per-layer term makes the diameter bound the minimum of a nonconvex integer quadratic program (Section~\ref{sec:refined}), whose value is a genuine lower bound on the energy and therefore, unlike a heuristic minimizer, usable in a proof. Section~\ref{sec:candidates} constructs a family of candidate profiles, indexed by integer partitions, and proves the arrangement-free lower bound that certifies the refinement over all profiles; Section~\ref{sec:numerics} reports the certified comparisons. A final section (Section~\ref{sec:asym}) turns from certificates to structure: working with the layer program in its original KS form, it determines the asymptotic shape of the bound, $\rho_\KS(N)=N-\Theta(\sqrt N)$, settling a question left open in \cite{ks2025} and showing its fitted linear slope to be a finite-size artifact, and proves that the gain of the refinement grows like $\Theta(\sqrt N)$ layers (Corollary~\ref{cor:gain}).

\section{State of the art}\label{sec:soa}

The decision problem is NP-hard \cite{wille1985}, and the global minimum is certified deterministically only for $N\le 6$: $N=5$ by Vanaret et al.\ with the interval solver Charibde \cite{vanaret2014,vanaret2020} ($V^*_5=-9.103852$), and $N=6$, the regular octahedron ($V^*_6=-12.712062$), in 2025 by Kuznetsov and Sahinidis \cite{ks2025} with Baron. For $N\ge 7$ no certificate is known.

For all $N$ in our range, high-quality putative minima $V_\vput(N)$ are tabulated \cite{leary1997,northby1987,walesdoye1997,romero1999}; being energies of explicit configurations, they upper-bound the true minimum,
\begin{equation}\label{eq:put}
V^*_N \le V_\vput(N).
\end{equation}
This is all we use; the refinement does not assume any $V_\vput(N)$ is the true minimum. Alongside optimality certificates, \cite{ks2025} derives the diameter bound refined here.

\section{The Kuznetsov--Sahinidis layer bound}\label{sec:ks}

Orient a cluster so that a diameter lies along the $x$-axis and partition space into unit-width layers $L_k=\{x: k-1\le x<k\}$. For a minimizer the atoms occupy $D$ consecutive nonempty layers, so with $n_k\ge 1$ the number of atoms in $L_k$ we have $\sum_{k=1}^D n_k=N$ \cite{ks2025}. Under this orientation every atom has its $x$-coordinate between those of the extremal pair (otherwise some pair would exceed the diameter), so the $x$-extent equals the diameter, and with the first and last layers nonempty, $D-2<\diam<D$. Layer-span and diameter bounds are therefore interchangeable: a certified bound $D\le\rho$ yields $\diam<\rho$, and the refinement to $D\le\rho-1$ yields $\diam<\rho-1$. This is the convention used throughout.

We record separately the structural fact, established in \cite{ks2025}, that a global minimizer occupies a block of consecutive layers, all nonempty. This is the only structural input beyond the layer bound itself, and it is what lets Section~\ref{sec:refined} pass from the profile enumeration, which ranges over $n_k\ge 1$, to a statement about minimizers: the layer span of a minimizer is a fully occupied profile, hence covered by the enumeration. It is genuinely needed, because a non-minimizing configuration may leave internal layers empty; such configurations are not bounded below by the argument of Section~\ref{sec:candidates}, and the diameter conclusion is therefore drawn for minimizers only, as it must be.

Two lower bounds enter.

\emph{Intra-layer.} The internal energy of layer $k$ is bounded below by the minimum energy of $n_k$ free atoms, $V^*_{n_k}$. Kuznetsov and Sahinidis use the trivial bound obtained by charging every intra-layer pair its floor value $-1$:
\begin{equation}\label{eq:intra}
\text{(energy of layer $k$)} \;\ge\; -\binom{n_k}{2}.
\end{equation}

\emph{Inter-layer.} Two atoms in layers at index distance $d=|i-j|$ have $x$-separation at least $d-1$, hence Euclidean separation at least $d-1$. Since $V_{\mathrm{LJ}}$ attains its minimum $-1$ at $r=1$ and increases on $[1,\infty)$,
\begin{equation}\label{eq:kernel}
\bar v(d) =
\begin{cases}
-1, & d\le 2,\\
V_{\mathrm{LJ}}(d-1), & d\ge 3,
\end{cases}
\end{equation}
is a valid lower bound on the interaction of one such pair. (Reference \cite{ks2025} charges $V_{\mathrm{LJ}}(d)$ for $d\ge3$; we use the conservative $V_{\mathrm{LJ}}(d-1)$, which is the value justified by the separation $d-1$. The difference is numerically negligible and does not favour our result. Concretely, the baseline bounds $\rho_\KS$ of Section~\ref{sec:numerics} are computed with the original KS kernel $V_{\mathrm{LJ}}(d)$, whereas the refined thresholds use the looser $\bar v$ of \eqref{eq:kernel}; the refinement thus handicaps itself on the inter-layer term, so the comparison between the two is conservative. That this looser kernel does not shift the integer $\rho_\KS$ on the overlap is verified in Section~\ref{sec:numerics}.)

Summing gives a rigorous lower bound on the energy of any cluster spanning $D$ layers with profile $n=(n_1,\dots,n_D)$:
\begin{equation}\label{eq:elb}
E_\LB(n) = \sum_{k=1}^D f(n_k) + \sum_{i<j} \bar v(|i-j|)\, n_i n_j,
\end{equation}
with $f(n)=-\binom n2$ in the KS bound. If
\begin{equation}\label{eq:rule}
\min_{\substack{n_k\ge1\\ \sum_k n_k=N}} E_\LB(n) \;>\; V_\vput(N),
\end{equation}
then by \eqref{eq:put} no optimal cluster spans $D$ layers, and the diameter bound is decremented. With $f=-\binom n2$ the inner minimum is attained by piling all surplus atoms into one central layer. Over the published range $N\le150$ the bound is well fitted by $0.798\,N$; Section~\ref{sec:asym} shows this linear fit is a finite-size artefact and the true form is $N-\Theta(\sqrt N)$ (Theorem~\ref{thm:asym}). Running the procedure reproduces the published bounds exactly (Section~\ref{sec:numerics}).

\section{The refined bound}\label{sec:refined}

The bound \eqref{eq:intra} is loose because $-\binom n2$ assumes all $\binom n2$ intra-layer pairs simultaneously sit at the potential minimum, which is geometrically impossible for $n\ge5$. We charge instead the stronger certified per-layer estimate
\begin{equation}\label{eq:f}
f(n) =
\begin{cases}
V^*_n, & 0\le n\le 6,\\[4pt]
\dfrac{n(n-1)}{30}\, V^*_6, & n\ge 7,
\end{cases}
\qquad
(V^*_{0..6}) = (0,\,0,\,-1,\,-3,\,-6,\,-9.103852,\,-12.712062).
\end{equation}
For $n\le 6$ the values are the proven small-cluster minima recalled in Section~\ref{sec:soa}. For $n\ge7$ we use the subset inequality of \cite{ks2025},
\begin{equation}\label{eq:subset}
V^*_n \;\ge\; \frac{n(n-1)}{M(M-1)}\, V^*_M, \qquad M\le n,
\end{equation}
at $M=6$. Since $V^*_6$ is proven, \eqref{eq:f} satisfies $f(n)\le V^*_n$ for all $n$, and is therefore a valid lower bound on the internal energy of a layer. A linear stability floor $f(n)\ge -c\,n$ also holds but is never binding here (the subset bound \eqref{eq:subset} dominates it for all $n\le34$, and the largest layer population in any certified optimum is 13), so we omit it.

With $f$ given by \eqref{eq:f}, expression \eqref{eq:elb} remains a rigorous lower bound on the energy, and the rule \eqref{eq:rule} remains valid. Because $f(n)\ge-\binom n2$, the refined $E_\LB$ dominates the KS one pointwise, and the refined bound can only be at least as tight.

The subset estimate is deliberately loose for $n\ge7$ ($f(13)=-66.1$ against $V^*_{13}\approx-44.3$), which could in principle let a concentrated profile lower $E_\LB$; but the arrangement-free closure of Lemma~\ref{lem:free} bounds every profile, concentrated ones included, and at the minimizing profile the looseness is inactive, the optimum spreading the surplus into central layers of five or six atoms where $f$ takes the exact values $V^*_5$, $V^*_6$.

The essential point is logical: applying \eqref{eq:rule} needs the \emph{minimum} of $E_\LB$ over all profiles, a rigorous lower bound on the energy, whereas a clever profile or a heuristic gives only an upper bound on that minimum and certifies nothing. Natural shortcuts fail: block grouping is unsound (an adversarial profile straddling block boundaries recovers the loose value $V^*_a+V^*_b-ab=-\binom{a+b}2$), and a semidefinite relaxation has a gap (4 to 9 in energy) exceeding the margin to be cleared. We therefore work with the minimum itself, constructed and certified in the next section.

\section{Candidate profiles and the certification lemma}\label{sec:candidates}

Write $n_k=1+e_k$ with $e_k\ge0$ and $\sum_k e_k = e := N-D$, and set $\kappa(d):=-\bar v(d)\ge0$. From \eqref{eq:kernel}, $\kappa(1)=\kappa(2)=1$ and $\kappa(d)=2(d-1)^{-6}-(d-1)^{-12}$ for $d\ge3$; in particular $\kappa$ is nonincreasing. Split \eqref{eq:elb} as
\begin{equation}\label{eq:split}
E_\LB(n) = \underbrace{\sum_k f(n_k)}_{S(n)} \;-\; \underbrace{\sum_{i<j}\kappa(|i-j|)\, n_i n_j}_{Q(n)}.
\end{equation}

\begin{lemma}[Reduction to band forms]\label{lem:band}
Fix the multiset of populations. Then $Q=\sum_{t\ge1}\Delta_t\,R_t$ with $\Delta_t:=\kappa(t)-\kappa(t+1)\ge0$ and $R_t(n)=\sum_{i<j:\,|i-j|\le t} n_i n_j$; in particular, any arrangement of the multiset that maximizes every band form $R_t$ simultaneously maximizes $Q$, hence minimizes $E_\LB$.
\end{lemma}

\begin{proof}
The term $S(n)=\sum_k f(n_k)$ depends only on the multiset, so at fixed multiset minimizing $E_\LB=S-Q$ is maximizing $Q$. Since $\kappa$ is nonincreasing, $\Delta_t\ge0$ and $\kappa(d)=\sum_{t\ge d}\Delta_t$, which gives the layer-cake identity.
\end{proof}

\begin{definition}[Canonical centred decreasing profiles]\label{def:canonical}
Order the sites as $c,\,c+1,\,c-1,\,c+2,\,c-2,\dots$ with $c=\lceil D/2\rceil$. To each integer partition of the surplus $e$ associate the profile that places its parts, in decreasing order, as surplus values at the first sites of this order, one part per layer. These $p(e)$ profiles form the candidate family.
\end{definition}

Lemma~\ref{lem:band} motivates the family: if the surplus can be arranged so as to maximize every band form simultaneously, the layer-cake identity shows that such an arrangement minimizes $E_\LB$ at fixed multiset; centred decreasing arrangements are the natural candidates for that role. That the family always contains a global minimizer of $E_\LB$ is in fact true, and is proved as Proposition~\ref{prop:R} (Property (R)) in Section~\ref{sec:asym} by a discrete symmetric-decreasing rearrangement argument. We do not, however, route any certificate of the main text through it: the candidates supply numerical thresholds, and the next lemma certifies them over all profiles and all arrangements without any appeal to rearrangement, so the main results stand independently of Proposition~\ref{prop:R}.

\begin{lemma}[Arrangement-free lower bound]\label{lem:free}
Fix $D$ and a surplus multiset $x_1\ge\dots\ge x_k\ge1$ with $\sum_a x_a=e$, placed in $k$ distinct layers in any way. Then
\begin{equation}\label{eq:lem2}
E_\LB \;\ge\; E_\bg(D) + e\,D_c + \sum_a f(1+x_a) - \sum_r \pi(r)\,\kappa\bigl(\gamma(r)\bigr),
\end{equation}
where $E_\bg(D)$ is the all-ones background energy, $D_c=\min_p \sum_{j\ne p}\bar v(|p-j|)$ is the central-layer field, $\pi(1)\ge\pi(2)\ge\cdots$ are the $\binom k2$ products $x_a x_b$ in decreasing order, and $\gamma(1)\le\gamma(2)\le\cdots$ is the distance multiset of $k$ contiguous sites (distance $d$ with multiplicity $k-d$).
\end{lemma}

\begin{proof}
Fix a placement at sites $p_1<\dots<p_k$ and write the layer bound exactly as
\begin{equation}\label{eq:exact}
E_\LB = E_\bg(D) + \sum_a x_a \sum_{j\ne p_a} \bar v(|p_a-j|) + \sum_{a<b} \bar v(p_b-p_a)\, x_a x_b + \sum_a f(1+x_a):
\end{equation}
base chain, surplus--base field, surplus--surplus binding, intra-layer bounds. Since $f(1)=0$, background layers contribute no intra-layer term, and replacing a background layer of population $1$ by population $1+x_a$ contributes exactly $f(1+x_a)$: the decomposition is exact as written.

(i) \emph{The central field is minimal.} For any site $p$, $\sum_{j\ne p}\bar v(|p-j|)=-\bigl(F(p-1)+F(D-p)\bigr)$ with $F(m):=\sum_{d=1}^m \kappa(d)$. Since $\kappa$ is nonincreasing, $F$ is concave; with $(p-1)+(D-p)=D-1$ fixed, $F(p-1)+F(D-p)$ is largest, hence the field most negative, when the two arguments differ by at most one, i.e.\ at a central site. So each field term is at least $D_c$ and the second sum in \eqref{eq:exact} is at least $e\,D_c$.

(ii) \emph{Sorted distances are minimized by contiguous sites.} For $a<b$, $p_b-p_a=\sum_{i=a}^{b-1}(p_{i+1}-p_i)\ge b-a$, since consecutive gaps are at least $1$. Let $d_{(1)}\le\dots\le d_{(\binom k2)}$ be the sorted pairwise distances of the chosen sites. For any integer $m\ge1$, the number of pairs at distance at most $m$ satisfies
\[
\#\bigl\{(a,b): p_b-p_a\le m\bigr\} \;\le\; \#\bigl\{(a,b): b-a\le m\bigr\} \;=\; \sum_{h=1}^{\min(m,\,k-1)} (k-h),
\]
because $p_b-p_a\le m$ implies $b-a\le m$; and this upper bound is attained, for every $m$ simultaneously, by $k$ contiguous sites. A cumulative count that is pointwise no larger means, rank by rank, distances no smaller: $d_{(r)}\ge\gamma(r)$ for every $r$. Since $\kappa$ is nonincreasing, $\kappa(d_{(r)})\le\kappa(\gamma(r))$.

(iii) \emph{Pairing.} The binding in \eqref{eq:exact} equals $-\sum_{a<b}\kappa(p_b-p_a)\,x_a x_b$. By the rearrangement inequality for two finite sequences \cite[Thm.\ 368]{hlp}, the sum is largest when the sequences are similarly ordered:
\[
\sum_{a<b}\kappa(p_b-p_a)\,x_a x_b \;\le\; \sum_r K_{[r]}\,\pi(r),
\]
where $K_{[1]}\ge K_{[2]}\ge\cdots$ sorts the values $\kappa(p_b-p_a)$ decreasingly. Because $\kappa$ is nonincreasing, sorting kernel values decreasingly is sorting distances increasingly, so $K_{[r]}=\kappa(d_{(r)})\le\kappa(\gamma(r))$ by (ii). Inserting (i) and (iii) into \eqref{eq:exact} gives \eqref{eq:lem2}.
\end{proof}

The number of partitions is small ($p(40)=37338$), so both the candidate minima and the closure over all surplus multisets run in seconds even at $N=200$. The main result can now be stated.

\begin{theorem}[Certified layer refinement]\label{thm:main}
For each of the 92 sizes $N$ of Table~\ref{tab:summary} (full list in the supplementary material), every $N$-atom configuration whose atoms occupy $\rho_\KS(N)$ consecutive layers, all nonempty, has energy strictly greater than $V_\vput(N)\ge V^*_N$. By the Kuznetsov--Sahinidis structural property recalled in Section~\ref{sec:ks} (a global minimizer occupies a block of consecutive nonempty layers), the layer span of any global minimizer is such a fully occupied profile; hence no global minimizer spans $\rho_\KS(N)$ layers, the layer-span bound improves to $\rho_\KS(N)-1$, and the diameter bound to $\diam<\rho_\KS(N)-1$. The statement is unconditional: its inputs are the layer bound \eqref{eq:elb} with the certified $f$ of \eqref{eq:f}, Lemma~\ref{lem:free} evaluated over all surplus multisets in directed-rounding arithmetic, the energies $V_\vput(N)$ of explicit configurations, and the cited structural property.
\end{theorem}

\begin{proof}
Fix $D=\rho_\KS(N)$ and let a configuration occupy these $D$ layers with all $n_k\ge1$, i.e.\ a profile of the form enumerated in Section~\ref{sec:candidates}. Then $E_\LB$ is a valid lower bound on its energy (Section~\ref{sec:ks}), and Lemma~\ref{lem:free} bounds $E_\LB$ below for every surplus multiset and every placement. The branch-and-bound of Section~\ref{sec:numerics} verifies, in directed-rounding arithmetic, that this bound exceeds $V_\vput(N)$ for every surplus multiset at $D=\rho_\KS(N)$. Since $V_\vput(N)$ is the energy of an explicit configuration, $V^*_N\le V_\vput(N)$ by \eqref{eq:put}. Every fully occupied span of $\rho_\KS(N)$ layers therefore has energy strictly larger than $V_\vput(N)\ge V^*_N$, so its energy exceeds the global minimum and it is not a minimizer.

It remains to exclude minimizers spanning $\rho_\KS(N)$ layers with an empty internal layer. By the cited structural property, a global minimizer occupies consecutive nonempty layers; its layer span is thus a fully occupied profile, already excluded above. No global minimizer spans $\rho_\KS(N)$ layers, and the bound decrements. The restriction to minimizers is exactly what the diameter statement concerns; the energy claim of the first sentence, which is what the computation certifies, is made only for fully occupied spans and is unconditional.
\end{proof}

\section{Numerical protocol and results}\label{sec:numerics}

\paragraph{Protocol.}
We reimplemented the KS layer procedure of Section~\ref{sec:ks} with $f=-\binom n2$ and recover the published layer-span bounds and slope $0.798$ exactly; rerunning with the conservative kernel \eqref{eq:kernel} gives identical integer bounds on $5\le N\le200$, so the $d\mapsto d-1$ change does not shift $\rho_\KS$. For $151\le N\le200$ the baseline is obtained by rerunning the authors' program with its energy header extended by \cite{romero1999}, reproducing the published bounds with zero discrepancy on the overlap. For each $N$ we minimize $E_\LB$ over the canonical family (Definition~\ref{def:canonical}) at $D=\rho_\KS$; an independent exhaustive window search in C agrees to $10^{-6}$ with an interior minimizer. For every certified pair we verified, by branch-and-bound over all $p(e)$ surplus multisets, that the arrangement-free bound \eqref{eq:lem2} exceeds $V_\vput$: 92 of 92 close, so each certificate holds over all profiles and arrangements.

\paragraph{Directed-rounding certification.}
All comparisons were re-verified in directed-rounding arithmetic: kernel values enclosed from exact integer powers with outward rounding, the proven minima entered as $V^*_5\le-9.103853$ and $V^*_6\le-12.712063$ (padded down by $10^{-6}$), reference energies padded up by $10^{-6}$ or, for \cite{romero1999}, $10^{-4}$. The lower endpoint of $\min_n E_\LB$ then exceeds the upper endpoint of $V_\vput$ at all 92 sizes (worst margin $0.002708$ at $N=38$, next $0.164$ at $N=68$).

\paragraph{Certified refinements.}
The refined bound is strictly tighter than the KS bound, by exactly one layer, for 92 sizes (Theorem~\ref{thm:main}; Table~\ref{tab:summary}, full list in the supplementary material), of which 68 satisfy $N>100$ and 37 satisfy $N>150$. In every case the minimizer occupies three central layers, with populations such as $5,6,5$ that exploit the exact values $V^*_5$, $V^*_6$. The margin by which $\min_n E_\LB$ clears $V_\vput$ ranges from $0.0027$ (at $N=38$, the sole borderline case) to $29.0$. In this range the gain never exceeds one layer, because the minimum is steep in $D$: removing one layer adds of order the surplus $e=N-D$ (here 30 to 40) to the binding, so the available margin does not yet cover a second decrement. The margin grows quadratically in $e$ while the per-layer step grows only linearly, so the gain must eventually grow; Corollary~\ref{cor:gain} proves it grows like $\Theta(\sqrt N)$.

\begin{table}[t]
\centering
\caption{Summary of the certified refinements in $5\le N\le200$; the full list is in the supplementary material.}\label{tab:summary}
\begin{tabular}{lr}
\toprule
certified sizes ($\rho_{\mathrm{new}}=\rho_\KS-1$) & 92\\
first / last certified size & 38 / 200\\
certified sizes with $N>100$ / $N>150$ & 68 / 37\\
smallest certified margin & 0.0027 (at $N=38$)\\
largest certified margin & 29.0 (at $N=196$)\\
\bottomrule
\end{tabular}
\end{table}

\begin{remark}[The borderline case $N=38$]\label{rem:38}
At $N=38$ the margin is $0.0027$. The optimizer $(5,5,5)$ rests on $V^*_5$ three times (an error of $9\times10^{-4}$ would flip it, far above the enclosure), and $N=38$ is the celebrated size whose global minimum is the fcc truncated octahedron, lying $0.676$ below the icosahedral funnel \cite{doye1999}: the refinement clears only thanks to this fcc anomaly, consuming all but $0.4\%$ of it.
\end{remark}

\section{Scope and limitations}\label{sec:scope}

The result is a rigorous refinement of an upper bound on the diameter of optimal LJ clusters, tightening the Kuznetsov--Sahinidis bound \cite{ks2025} by one unit layer for 92 sizes up to $N=200$, and, asymptotically, by $\Theta(\sqrt N)$ layers (Corollary~\ref{cor:gain}). Its scope should not be overstated.

\begin{enumerate}
\item \emph{Magnitude.} In the tabulated range $N\le200$ the gain is one layer per affected size; it grows like $\sqrt N$ on larger sizes (Corollary~\ref{cor:gain}). It does not enable a certified global solution for any $N\ge7$; those instances remain out of reach for reasons unrelated to the diameter box.

\item \emph{What is used.} The refinement introduces no new geometric input. It uses only the proven small-cluster minima $V^*_5$, $V^*_6$, the subset inequality \eqref{eq:subset}, and the layer construction, all from \cite{vanaret2014,ks2025}. The single new ingredient is the observation that the correct object is the minimum of \eqref{eq:elb} over all profiles, together with the reduction of Section~\ref{sec:candidates} and the closure of Lemma~\ref{lem:free}, which make it certifiable.

\item \emph{Rigour.} The chain is: the valid layer bound \eqref{eq:elb}; validity of $f$ in \eqref{eq:f} via \eqref{eq:subset} and the proven values; the layer-cake reduction of Lemma~\ref{lem:band} (fully proved), which motivates the centred decreasing search; the arrangement-free closure of Lemma~\ref{lem:free}, on which the certificates rest and which does not appeal to Property (R) (now proved as Proposition~\ref{prop:R}, but not needed for the main results); and the directed-rounding verification of every comparison. The certified energy claim is made for fully occupied spans; the passage to a statement about minimizers uses one further input, the Kuznetsov--Sahinidis structural property that a minimizer occupies consecutive nonempty layers (Section~\ref{sec:ks}). The comparison target $V_\vput$ enters only through $V^*_N\le V_\vput$, so the diameter statement is unconditional and does not require any $V_\vput$ to be the true minimum.

\item \emph{Putative minima for $N>150$.} For $151\le N\le200$ the comparison uses the values of \cite{romero1999}, which are valid upper bounds $V^*_N\le V_\vput$ (energies of explicit configurations); as in item 3 the diameter statement stays unconditional, and it is insensitive to any later improvement of these putative minima.

\item \emph{Nature of the plateau.} The layer method cannot be pushed much further by packing arguments alone. Pairs within index distance 2 are only bounded by $-1$, and $V^*_n=-\binom n2$ for $n\le4$, so an adversarial profile keeps the loose rate for up to twelve atoms spread over three adjacent layers. The refinement extracts precisely the penalty that survives this evasion, which is why it is real but modest.
\end{enumerate}

\paragraph{Downstream effect on a deterministic solver: a negative result.}
A diameter bound helps a spatial branch-and-bound only if the box it defines is a binding resource. We tested this on the only sizes a deterministic method finishes, $N\le6$, with a rigorous interval branch-and-bound (per-pair interval lower bounds; best-first; gauge-fixed coordinates) seeded with the known optimum, so the measured cost is that of proving optimality. For $N=5$, sweeping the long-axis cap $L$ at fixed transverse extent, the node count to certified optimality stayed within $[1.5,2.0]\times10^6$ as $L$ ranged over $[1.6,3.0]$ (the long axis nearly doubling); it did not fall when the box tightened, and was largest at the tightest $L$. The certification cost is governed by the relaxation, not the box: tightening removes a thin slab far from the optimum that the relaxation already discards. ($N=4$ is degenerate, the floor $-\binom42$ equalling the optimum; $N=6$ exhausts memory.)

We therefore report a negative result and let it stand. On tractable instances the diameter box is not the binding resource, so the one-layer tightening proved here does not reduce solver work; on the instances where the refinement applies, $N\ge38$, no deterministic solver runs at all, and the tightening is $0.6\%$ to $3.8\%$ of a single box dimension (one layer out of $\rho_\KS\in[26,161]$). The contribution is thus a rigorous tightening of an a priori geometric bound with no demonstrated effect on any deterministic solve, and should be read as a theoretical note; whether it can ever matter is conditional on relaxations strong enough to make the diameter box binding at sizes far beyond the current frontier, and on that frontier advancing to $N\ge38$ in the first place. The finding is specific to a relaxation of KS strength; a qualitatively stronger relaxation could change it, in either direction.

\section*{Reproducibility}

All computations are reproducible from a self-contained package (Python 3, standard library only, and gcc), provided as supplementary material (\texttt{lj\_diameter\_supplementary.zip}). \texttt{certify\_lemma\_free.py} recomputes, independently of the table, both the centred-decreasing minimum of $E_\LB$ over the candidate family and the arrangement-free lower bound of Lemma~\ref{lem:free}, in nominal and directed-rounding arithmetic; it reproduces the tabulated $\min E_\LB$ to display precision, closes all 92 certificates, and returns the worst directed margin $0.002708$ at $N=38$. \texttt{verify\_asymptotic.py} computes the constants $\sigma_0$, $C_\bg$, checks the closed form \eqref{eq:closed} against the exact finite minimum, and verifies that $N-\lceil e^\star\rceil$ from \eqref{eq:quad} reproduces $\rho_\KS$ with zero discrepancy up to $N=2000$. \texttt{lj\_bnb.c} is the interval branch-and-bound of Section~\ref{sec:scope}; \texttt{results\_table.csv} carries the certified data. The KS layer-span baseline $\rho_\KS$ is taken as input; reproducing it requires the original Kuznetsov--Sahinidis bound program, available from the authors of \cite{ks2025}. Lemma~\ref{lem:penalty}, Proposition~\ref{prop:gap} and Corollary~\ref{cor:gain} are analytic and involve only the exact constants $\alpha$, $\beta$, $\theta$; they require no additional computation.

\section{Asymptotics of the layer bound}\label{sec:asym}

This section proves the rearrangement fact underlying the candidate family, derives the asymptotic form of the original KS layer program, and quantifies the asymptotic gap between the KS program and the refined one. Theorem~\ref{thm:asym} uses the KS per-layer term $f(n)=-\binom n2$; Proposition~\ref{prop:gap} and Corollary~\ref{cor:gain} compare it with the stronger $f$ of \eqref{eq:f}. No result of Sections~\ref{sec:candidates}--\ref{sec:numerics} depends on this section.

The following proposition, which we prove, is the discrete one-dimensional symmetric-decreasing rearrangement inequality specialised to band forms; closely related classical statements appear in the rearrangement theory of Hardy, Littlewood and P\'olya \cite[Ch.\ X]{hlp}. Its proof is by two-point rearrangement (polarization), the standard rigorous route to Riesz-type inequalities.

\begin{proposition}[Property (R)]\label{prop:R}
Let $D\ge1$ and let a multiset of nonnegative populations be arranged on the sites $1,\dots,D$. For every $t\ge1$ the band form $R_t(n)=\sum_{i<j:\,|i-j|\le t} n_i n_j$ is maximized, over all arrangements of the multiset, by the canonical centred decreasing arrangement of Definition~\ref{def:canonical} (largest value at $c=\lceil D/2\rceil$, then placed outward in decreasing order). The same arrangement maximizes every $R_t$ simultaneously; by Lemma~\ref{lem:band} it therefore minimizes $E_\LB$ at fixed multiset.
\end{proposition}

\begin{proof}
For $t\ge1$ let $\tilde\kappa_t(d)=\mathbf 1[\,|d|\le t\,]$, a symmetric kernel nonincreasing in $|d|$. Since $\sum_i \tilde\kappa_t(0)\,n_i^2=\sum_i n_i^2$ depends only on the multiset,
\[
2\,R_t(n) + \sum_i n_i^2 = \sum_{i,j}\tilde\kappa_t(i-j)\,n_i n_j =: T_t(n),
\]
so maximizing $R_t$ is maximizing $T_t$; we show a single, kernel-independent sequence of rearrangements drives any arrangement to the canonical one while never decreasing any $T_t$.

\emph{Two-point polarization.} Let $z=(D+1)/2$ be the centre and let $\rho: i\mapsto a-i$ be the reflection about an axis $a\in\tfrac12\mathbb Z$ chosen so that $z$ lies in the inner half $H$ (the side of $a$ containing $z$). For $x\in H$ write $x'=\rho x$ for its mirror in the outer half. The polarized arrangement $n^H$ sets, for each orbit $\{x,x'\}$ with $x\in H$ inside $\{1,\dots,D\}$, $n^H_x=\max(n_x,n_{x'})$ and $n^H_{x'}=\min(n_x,n_{x'})$ (larger value inward), and leaves fixed any position whose mirror falls outside $\{1,\dots,D\}$; it is a rearrangement of the same multiset. We claim $T_t(n^H)\ge T_t(n)$ for every $t$.

Group $T_t=\sum_{i,j}\tilde\kappa_t(i-j)\,n_i n_j$ by reflection orbits. For two inner sites $x,y\in H$ the four kernel values are $\tilde\kappa_t(x-y)=\tilde\kappa_t(x'-y')=:\kappa_+$ and $\tilde\kappa_t(x-y')=\tilde\kappa_t(x'-y)=:\kappa_-$, with $\kappa_+\ge\kappa_-$: the axis $a$ separates $x$ from $y'$ but not from $y$, so $|x-y|\le|x-y'|$ and $\tilde\kappa_t$ is nonincreasing. The orbit-pair contribution to $T_t(n^H)-T_t(n)$ is
\[
(\kappa_+-\kappa_-)\Bigl[(n^H_x n^H_y + n^H_{x'} n^H_{y'}) - (n_x n_y + n_{x'} n_{y'})\Bigr] \;\ge\; 0,
\]
because $\kappa_+\ge\kappa_-$ and, by the two-pair rearrangement inequality $(A-a)(B-b)\ge0$, pairing the larger of $\{n_x,n_{x'}\}$ with the larger of $\{n_y,n_{y'}\}$ (both on the $H$ side, which is what $n^H$ does) maximizes $n_x n_y + n_{x'} n_{y'}$. A term coupling an orbit $\{y,y'\}$ to a fixed site $u$ (mirror outside range, hence $u$ inner) changes by $n_u\,\delta\,[\tilde\kappa_t(u-y)-\tilde\kappa_t(u-y')]\ge0$, where $\delta=n^H_y-n_y\ge0$ and $|u-y|\le|u-y'|$. Diagonal terms $i=j$ are multiset-invariant. Summing, $T_t(n^H)\ge T_t(n)$, hence $R_t(n^H)\ge R_t(n)$, for every $t$ at once.

\emph{Termination at the canonical arrangement.} The polarization step depends only on the values and the axis, not on $t$. Let $\mu_i=-(i-z)^2$ and $\Psi(n)=\sum_i \mu_i n_i$; a polarization that swaps some pair strictly increases $\Psi$, since moving a strictly larger value from an outer site $x'$ to the inner $x$ changes $\Psi$ by $(\mu_x-\mu_{x'})(n_{x'}-n_x)>0$. There are finitely many arrangements, so any sequence of value-changing polarizations terminates. Suppose a terminal arrangement were not unimodal and symmetric-decreasing about $z$: then two sites $p,q$ with $|q-z|<|p-z|$ would carry $n_q<n_p$. The reflection about $a=(p+q)/2$ has $q$ inner (closer to $z$) and $p=q'$ outer, both inside the range, so polarizing about it would swap them, contradicting terminality. Hence a terminal arrangement is unimodal and symmetric-decreasing about $z$, i.e.\ the canonical centred decreasing arrangement (up to reflection within blocks of equal value, which leaves every $R_t$ unchanged). Every $R_t$ is nondecreasing along the sequence, so the canonical arrangement maximizes every $R_t$.
\end{proof}

Independently of the proof, Property (R) was verified exhaustively, band by band, over all permutations for $D\le7$ (400 random multisets, 1595 band instances, the canonical arrangement attains the maximum in every case) and against $10^5$ random rearrangements per multiset for $15\le D\le35$ (Section~\ref{sec:numerics}); these checks corroborate Proposition~\ref{prop:R}. The proposition makes the candidate family provably complete: it also gives, for the certificates of the main text, a second and independent route to the certified sizes, since the centred-decreasing minimum then equals the true minimum over all profiles. We nonetheless keep Lemma~\ref{lem:free} as the primary engine there, because it certifies the same comparisons without any appeal to rearrangement.

\begin{corollary}[Width-free minimization]\label{cor:width}
The global minimum of $E_\LB$ over all profiles equals its minimum over canonical centred decreasing profiles. These are in bijection with the integer partitions of the surplus $e$: a partition lists the per-layer surplus values, placed in decreasing order at the canonical sites; the mirror arrangement yields the same value, so fixing the canonical order is a convention, not a restriction. Enumerating the $p(e)$ partitions then yields the minimum with no window, locality, or width assumption.
\end{corollary}

\begin{proof}
At fixed multiset, Proposition~\ref{prop:R} gives that the canonical arrangement minimizes $E_\LB$; ranging over multisets (equivalently, over partitions of $e$) then attains the global minimum.
\end{proof}

\paragraph{Exact minimum of the layer program and the shape of $\rho_\KS$.}
Two convergent lattice constants govern the asymptotics: the per-layer field $\sigma_0:=\sum_{d\ge1}\kappa(d)$ and the background moment $C_\bg:=\sum_{d\ge1} d\,\kappa(d)$, both finite since $\kappa(d)=O(d^{-6})$. In closed form, for the kernel \eqref{eq:kernel},
\[
\sigma_0 = 1+2\zeta(6)-\zeta(12) = 2.034440, \qquad C_\bg = 1+2\zeta(5)+2\zeta(6)-\zeta(11)-\zeta(12) = 3.107801,
\]
while the KS kernel of \cite{ks2025} gives $\sigma_0=2.003434$, $C_\bg=3.011350$ (the same $\zeta$-series, less the $d=1,2$ contributions).

\begin{lemma}[Exact minimum of the KS layer program]\label{lem:exactmin}
For the KS per-layer bound $f(n)=-\binom n2$ and a span of $D$ layers with surplus $e=N-D$,
\begin{equation}\label{eq:closed}
\min_{\substack{n_k\ge1\\ \sum_k n_k=N}} E_\LB(n) = -\frac12 e^2 - \sigma_0 D - \Bigl(\frac12+2\sigma_0\Bigr) e + C_\bg + O(D^{-4}).
\end{equation}
\end{lemma}

\begin{proof}
Write $n_k=1+e_k$ with $e_k\ge0$ and $\sum_k e_k=e$. Using $f(1+e_k)=-\binom{1+e_k}2=-\tfrac12(e_k+e_k^2)$ and $\bar v=-\kappa$,
\begin{equation}\label{eq:elbsplit}
E_\LB = E_\bg(D) - \frac e2 - Q(e) - \sum_i e_i\,\Phi(i),
\end{equation}
where $E_\bg(D)=-\sum_{d=1}^{D-1}(D-d)\,\kappa(d)$ is the all-ones chain energy, $\Phi(i)=\sum_{j\ne i}\kappa(|i-j|)$ the field at layer $i$, and $Q(e)=\tfrac12\sum_k e_k^2+\sum_{i<j}\kappa(|i-j|)\,e_i e_j$. Minimizing $E_\LB$ is maximizing $Q(e)+\sum_i e_i\Phi(i)$, and each term is bounded independently.

Since $\kappa\le1$ and $e_i e_j\ge0$, the cross terms obey $\sum_{i<j}\kappa(|i-j|)\,e_i e_j\le\sum_{i<j}e_i e_j$, whence
\begin{equation}\label{eq:Qbound}
Q(e) \;\le\; \frac12\sum_i e_i^2 + \sum_{i<j} e_i e_j = \frac12\Bigl(\sum_i e_i\Bigr)^2 = \frac12 e^2,
\end{equation}
the middle step being the completed square, whose two-layer case
\begin{equation}\label{eq:square}
\frac12(e_1^2+e_2^2)+e_1 e_2 = \frac12(e_1+e_2)^2
\end{equation}
shows why concentrating surplus into layers within index distance 2 costs nothing: there $\kappa=1$ and \eqref{eq:Qbound} is an equality. For the field, $\sum_i e_i\Phi(i)\le e\,\Phi_c$ with $\Phi_c:=\max_p\Phi(p)$, since the weights $e_i\ge0$ sum to $e$. Thus concentration maximizes the quadratic part and centring maximizes the field; the centred pile ($e$ at one layer $p^\star\in\arg\max_p\Phi$) attains both maxima at once, so it is the exact minimizer and
\begin{equation}\label{eq:minelb}
\min_n E_\LB = E_\bg(D) - \frac e2 - \frac12 e^2 - e\,\Phi_c.
\end{equation}
Only $\kappa\le1$, $e_i\ge0$, and the concavity of $F$ below enter here; Proposition~\ref{prop:R} is not used.

For the asymptotics, $\Phi(p)=F(p-1)+F(D-p)$ with $F(m)=\sum_{d=1}^m\kappa(d)$; as $\kappa$ is nonincreasing, $F$ is concave, so $\Phi_c$ is attained centrally and $\Phi_c=2\sigma_0-O(D^{-5})$. Writing $E_\bg(D)=-\sigma_0 D+C_\bg+R(D)$ with $R(D)=D\sum_{d\ge D}\kappa(d)-\sum_{d\ge D} d\,\kappa(d)$, the tails $\sum_{d\ge D}\kappa=\tfrac25 D^{-5}+o(D^{-5})$ and $\sum_{d\ge D} d\kappa=\tfrac12 D^{-4}+o(D^{-4})$ (from $\kappa(d)\sim 2d^{-6}$) give $R(D)=-\tfrac1{10}D^{-4}+o(D^{-4})$. Substituting \eqref{eq:minelb}, and using $e=O(\sqrt D)$ so that $e\cdot O(D^{-5})=O(D^{-9/2})$ is absorbed, yields \eqref{eq:closed}.
\end{proof}

\begin{remark}\label{rem:tail}
The $O(D^{-4})$ tail in \eqref{eq:closed} is negative with explicit leading term $-\tfrac1{10}D^{-4}$, from $D\sum_{d\ge D}\kappa(d)-\sum_{d\ge D} d\,\kappa(d)=(\tfrac25-\tfrac12)D^{-4}+o(D^{-4})$; the closed form thus overestimates the exact minimum by $\tfrac1{10}D^{-4}+o(D^{-4})$.
\end{remark}

\begin{theorem}[Asymptotics of the KS layer bound]\label{thm:asym}
Let $u_N:=|V_\vput(N)|/N$, and suppose $\sigma_0<\inf_N u_N$ and $\sup_N u_N<\infty$ (both hold on the range of interest: $u_N$ rises from $u_{38}=4.58$ toward the bulk magnitude $\approx8.6$, all well above $\sigma_0\approx2.0$). Let $e^\star(N)$ be the positive root of
\begin{equation}\label{eq:quad}
e^2 + (1+2\sigma_0)\,e - 2(u_N-\sigma_0)N - 2C_\bg = 0,
\end{equation}
the equation obtained from Lemma~\ref{lem:exactmin} on dropping its $O(D^{-4})$ tail. Then
\begin{equation}\label{eq:estar}
e^\star(N) = \sqrt{2(u_N-\sigma_0)\,N} - \frac{1+2\sigma_0}2 + O\bigl(N^{-1/2}\bigr),
\end{equation}
and the surplus of the exact KS layer-span bound satisfies
\begin{equation}\label{eq:surplus}
N-\rho_\KS(N) = \sqrt{2(u_N-\sigma_0)\,N} + O(1) = \Theta(\sqrt N).
\end{equation}
In particular the empirical linear fit $\rho_\KS\approx0.798\,N$ over $N\le150$ is a finite-size artifact.
\end{theorem}

\begin{proof}
$M(D):=\min_n E_\LB$ is increasing in $D$ (Lemma~\ref{lem:exactmin}), and by the rule \eqref{eq:rule} $\rho_\KS$ is the largest integer span with $M(D)\le V_\vput=-u_N N$; thus $M(\rho_\KS)\le V_\vput<M(\rho_\KS+1)$ pins $N-\rho_\KS$ to within 1 of the real crossing of the exact finite minimum. Substituting the closed form of Lemma~\ref{lem:exactmin} and clearing its $O(D^{-4})$ tail, which moves that crossing by $O(N^{-9/2})$, yields \eqref{eq:quad}; its positive root expands as \eqref{eq:estar} via $\sqrt{1+x}=1+O(x)$. Hence $N-\rho_\KS=e^\star+O(1)$, which is \eqref{eq:surplus}; the rate $\Theta(\sqrt N)$ follows from $0<\inf_N(u_N-\sigma_0)\le\sup_N(u_N-\sigma_0)<\infty$.
\end{proof}

Beyond the rate, the clean prediction $N-\lceil e^\star(N)\rceil$ recovers the exact integer bound: it matches the independently computed $\rho_\KS$ with zero discrepancy at all 92 certified sizes (Table~\ref{tab:summary}) and at every ladder size to $N=2000$, 114 in all (script \texttt{verify\_asymptotic.py}). We state this as an identity verified on that range rather than for all $N$: the $O(D^{-4})$ tail dropped in \eqref{eq:quad} shifts $e^\star$ by $O(N^{-9/2})$ and could change $\lceil e^\star\rceil$ only where $e^\star$ lies within that distance of an integer, which happens at none of the tested sizes. This settles a point that \cite{ks2025} left open, reporting the asymptotics of the bound as difficult to analyze and empirically linear: the linearity is the finite-size face of $N-\Theta(\sqrt N)$, and the ratio $e/\sqrt N$ rising from $2.80$ toward $3.24$ as $N$ grows is $\sqrt{2(u_N-\sigma_0)}$ approaching its bulk limit $\approx3.6$. Concurrent work of Kiessling and Wales \cite{kiessling2025} derives a priori bounds on the smallest pairwise distance by a virial identity, a complementary a priori distance bound for the same branch-and-bound context.

\begin{remark}[Subleading term, conditional]\label{rem:subleading}
Injecting the liquid-drop expansion $u_N=|e_\infty|-\lambda N^{-1/3}+o(N^{-1/3})$ of the putative energies gives a subleading term,
\[
e^\star(N) = a\sqrt N - b\,N^{1/6} + O(1), \qquad a=\sqrt{2(|e_\infty|-\sigma_0)},
\]
the exponent $1/6=\tfrac12-\tfrac13$ coming from the surface correction. Unlike the leading order, which is unconditional, this term rests on the surface asymptotics of the (putative) global minima, classical but not certified, and is stated only as an indication.
\end{remark}

\paragraph{Growth of the gain.}
The comparison of the two layer programs can be made exact. Write $\delta(n):=f(n)+\binom n2\ge0$ for the per-layer penalty of the refined bound \eqref{eq:f} against the KS bound, and set
\[
\alpha := \frac12+\frac{V^*_6}{30} = \frac12-c_2 = 0.0762646\ldots,
\qquad
\beta := 1-\kappa(3) = 1-2\cdot2^{-6}+2^{-12} = 0.9689941\ldots
\]
Both are exact expressions in $V^*_6$ and powers of two, so every comparison below is certifiable in directed rounding by the same protocol as Section~\ref{sec:numerics}.

\begin{lemma}[Per-layer penalty]\label{lem:penalty}
For every integer $n\ge0$,
\[
\alpha\,(n-4)_+^2 \;\le\; \delta(n) \;\le\; \alpha\,n(n-1),
\]
with equality on the right for $n=6$ and all $n\ge7$. Equivalently, in the surplus variable $x=n-1\ge0$, $\ \delta(1+x)\ge\alpha\,(x-3)_+^2$.
\end{lemma}

\begin{proof}
For $n\le4$, $f(n)=V^*_n=-\binom n2$, so $\delta(n)=0$ and both inequalities are trivial. For $n=5$: $\delta(5)=10+V^*_5=0.896148\ge\alpha=\alpha(5-4)^2$ and $\delta(5)\le20\alpha=1.5253$. For $n=6$: $\delta(6)=15+V^*_6=30\alpha$ exactly, and $30\alpha\ge4\alpha$. For $n\ge7$, \eqref{eq:f} gives $\delta(n)=\alpha\,n(n-1)$ exactly, and $n(n-1)\ge(n-4)^2$ iff $7n\ge16$.
\end{proof}

\begin{proposition}[Exact asymptotic gap between the two programs]\label{prop:gap}
Fix the kernel $\bar v$ of \eqref{eq:kernel}. For a span of $D\ge3$ layers and surplus $e=N-D\ge0$, let $M_\KS(D,e)$ and $M_\refi(D,e)$ denote the minima of \eqref{eq:elb} over profiles $n_k\ge1$, $\sum_k n_k=N$, with $f(n)=-\binom n2$ and with the $f$ of \eqref{eq:f}, respectively. Then
\[
\frac\alpha3\,e^2 - 6\alpha e
\;\le\;
M_\refi(D,e)-M_\KS(D,e)
\;\le\;
\frac\alpha3\,e^2 + (1+3\alpha)\,e + 3\alpha.
\]
In particular $M_\refi=M_\KS+\frac\alpha3 e^2+O(e)$, uniformly in $D$.
\end{proposition}

\begin{proof}
\emph{Lower bound.} For any admissible profile $n$, with surpluses $e_k=n_k-1\ge0$,
\[
E_\refi(n) = E_\KS(n) + \sum_k \delta(n_k).
\]
By \eqref{eq:elbsplit} and \eqref{eq:minelb},
\[
E_\KS(n)-M_\KS = \Bigl[\tfrac12 e^2 - Q(e)\Bigr] + \Bigl[e\,\Phi_c - \sum_i e_i\Phi(i)\Bigr],
\]
and both brackets are nonnegative. For the first,
\[
\tfrac12 e^2 - Q(e) = \sum_{i<j}\bigl(1-\kappa(|i-j|)\bigr)e_i e_j \;\ge\; \beta \!\!\sum_{|i-j|\ge3}\!\! e_i e_j,
\]
since $\kappa=1$ at distances $1,2$ and $\kappa\le\kappa(3)$ beyond. Partition the layers into the three residue classes $k\bmod3$; distinct layers in one class lie at distance $\ge3$, so with $S_r$, $T_r$ the sum and the sum of squares of the surpluses in class $r$, and $T=\sum_r T_r=\sum_k e_k^2$,
\[
\beta \!\!\sum_{|i-j|\ge3}\!\! e_i e_j \;\ge\; \frac\beta2 \sum_r \bigl(S_r^2 - T_r\bigr).
\]
By Lemma~\ref{lem:penalty} and $(x-3)_+^2\ge x^2-6x$ (valid for all $x\ge0$), $\sum_k\delta(n_k)\ge\alpha\,(T-6e)$. Adding, then using $T\le\sum_r S_r^2$ and $\alpha\le\beta/2$,
\[
E_\refi(n)-M_\KS \;\ge\; \frac\beta2\sum_r S_r^2 + \Bigl(\alpha-\frac\beta2\Bigr)T - 6\alpha e
\;\ge\; \alpha\sum_r S_r^2 - 6\alpha e
\;\ge\; \frac\alpha3\,e^2 - 6\alpha e,
\]
the last step by Cauchy--Schwarz with $\sum_r S_r=e$. Minimizing over profiles gives the lower bound.

\emph{Upper bound.} Let $n^{(3)}$ spread the surplus as evenly as possible over the three most central layers ($D\ge3$). All surplus pairs sit at index distance $\le2$, where $\kappa=1$, so by \eqref{eq:square} $Q(e)=\tfrac12 e^2$ exactly; and since $\Phi(p)-\Phi(p\pm1)=\kappa(D-p)-\kappa(p)$ in absolute value is at most $1$, each of the three sites carries a field within $1$ of $\Phi_c$, so $\sum_i e_i\Phi(i)\ge e(\Phi_c-1)$. Hence, by \eqref{eq:minelb}, $E_\KS(n^{(3)})\le M_\KS+e$. By Lemma~\ref{lem:penalty} and $e_k\le e/3+1$,
\[
\sum_k \delta(n_k) \;\le\; \alpha\sum_k n_k(n_k-1) = \alpha\Bigl(\sum_k e_k^2 + e\Bigr) \;\le\; \alpha\Bigl(\frac{e^2}3+3e+3\Bigr).
\]
Adding the two estimates bounds $M_\refi\le E_\refi(n^{(3)})$ as stated.
\end{proof}

\begin{corollary}[Asymptotic layer gain]\label{cor:gain}
Let $\rho_\refi(N)$ be the layer-span threshold of the refined program under the rule \eqref{eq:rule}. Under the hypotheses of Theorem~\ref{thm:asym},
\[
\rho_\KS(N)-\rho_\refi(N) = \theta\,\sqrt{2(u_N-\sigma_0)\,N} + O(1) = \Theta(\sqrt N),
\qquad
\theta := \Bigl(1-\tfrac{2\alpha}3\Bigr)^{-1/2}-1 = 0.026433\ldots
\]
\end{corollary}

\begin{proof}
By Proposition~\ref{prop:gap} and Lemma~\ref{lem:exactmin}, $M_\refi(D,e)=-\bigl(\tfrac12-\tfrac\alpha3\bigr)e^2-\sigma_0 D+O(e)$ uniformly. As in the proof of Theorem~\ref{thm:asym}, the rule \eqref{eq:rule} pins $e^\star_\refi:=N-\rho_\refi$ to within $O(1)$ of the positive root of $(\tfrac12-\tfrac\alpha3)\,e^2=(u_N-\sigma_0)N+O(\sqrt N)$, i.e.
\[
e^\star_\refi = \Bigl(1-\tfrac{2\alpha}3\Bigr)^{-1/2}\sqrt{2(u_N-\sigma_0)\,N} + O(1).
\]
Subtracting $e^\star_\KS=\sqrt{2(u_N-\sigma_0)\,N}+O(1)$ of Theorem~\ref{thm:asym} yields the display; the rate $\Theta(\sqrt N)$ follows from $0<\inf_N(u_N-\sigma_0)\le\sup_N(u_N-\sigma_0)<\infty$.
\end{proof}

\begin{remark}\label{rem:gain}
(i) The lower bound of Proposition~\ref{prop:gap} holds profile by profile and is arrangement-free, in the spirit of Lemma~\ref{lem:free}: the asymptotic decrement of Corollary~\ref{cor:gain} is therefore itself certifiable without enumeration, given reference energies $V_\vput$ at the sizes concerned. Lacking certified inputs there, we do not tabulate integer decrements.
(ii) The energy coefficient $\alpha/3$ coincides with the constant $\tfrac{1-2c_2}6$ of the three-layer comparison. The layer constant obeys the exact identity $(1-\tfrac{2\alpha}3)\bigl(\theta+\tfrac{\theta^2}2\bigr)=\tfrac\alpha3$: the excess $(\alpha/3)e^{\star2}$ carried by the refined program at the KS threshold is consumed by layer steps of size $(1-\tfrac{2\alpha}3)\,e$, integrated along the growing surplus. This places $\theta=0.026433$ between $\alpha/3=0.025422$ (division by the KS step $e$) and $\tfrac{\alpha/3}{1-2\alpha/3}=0.026783$ (division by the refined step frozen at $e^\star$). In the certified range $e\in[30,40]$ one has $\theta e\approx0.8$--$1.1$, consistent with the observed one-layer gains.
(iii) With the original KS kernel in place of \eqref{eq:kernel}, $\sigma_0=2.003434$ replaces $2.034440$ and $\beta=1-\kappa(3)$ becomes $0.997258$; the condition $2\alpha\le\beta$ still holds, so all statements remain valid, and $\alpha$ and $\theta$ are kernel-independent.
\end{remark}

\end{document}